\documentclass[11pt, oneside]{amsart}
\usepackage[text={5.58in,8.5in},centering,letterpaper,dvips,headheight=25pt]{geometry}
\usepackage[dvipsnames]{xcolor}
\usepackage{graphicx}
\usepackage{subcaption}
\usepackage{amsfonts}
\usepackage{epsf}
\usepackage{amssymb}
\usepackage{amsmath}
\usepackage{amscd}
\usepackage{tikz}
\usepackage{pdfpages}

\usepackage{fancyhdr}
\usepackage{setspace}
\usepackage{hyperref}
\usepackage[all]{xy}
\usetikzlibrary{matrix}
\usepackage{verbatim}
\usepackage{enumerate}

\newtheorem{theorem}{Theorem}[section]
\newtheorem{proposition}[theorem]{Proposition}
\newtheorem{lemma}[theorem]{Lemma}
\newtheorem{question}[theorem]{Question}
\newtheorem{corollary}[theorem]{Corollary}

\theoremstyle{definition}
\newtheorem{definition}[theorem]{Definition}

\newcommand{\be}{\begin{enumerate}}
\newcommand{\ee}{\end{enumerate}}

\makeatletter
\def\@seccntformat#1{%
  \protect\textup{\protect\@secnumfont
    \ifnum\pdfstrcmp{subsection}{#1}=0 \bfseries\fi% subsection # in \bfseries
    \csname the#1\endcsname
    \protect\@secnumpunct
  }
}  
\makeatother

\makeatletter
\newtheorem*{rep@theorem}{\rep@title}
\newcommand{\newreptheorem}[2]{%
\newenvironment{rep#1}[1]{%
 \def\rep@title{#2 \ref{##1}}%
 \begin{rep@theorem}}%
 {\end{rep@theorem}}}
\makeatother

\newreptheorem{theorem}{Theorem}
\newreptheorem{lemma}{Lemma}
\newreptheorem{openquestion}{Open Question}
\newreptheorem{corollary}{Corollary}
\newreptheorem{proposition}{Proposition}

\begin{document}

\raggedbottom
\pagenumbering{arabic}
\setcounter{section}{0}

%%%%%%%%%%%%%%%%%%%%%%%%%%%%%%%%%%%%%%%%%%%%%%%%%%%%%%%%
%%%%%%%%%%%%%%%%%%%%%%%%%%%%%%%%%%%%%%%%%%%%%%%%%%%%%%%%
%%%%%%%%%%%%%%%%%%%%%%%%%%%%%%%%%%%%%%%%%%%%%%%%%%%%%%%%

\title{Gordian Distance and Complete Alexander Neighbors}

\author{Ana Wright}
\address{Department of Mathematics, University of Nebraska-Lincoln, Lincoln, NE 68588}
\email{awright@huskers.unl.edu}
\urladdr{https://anawright.github.io}

\begin{abstract}
    We call a knot $K$ a \emph{complete Alexander neighbor} if every possible Alexander polynomial is realized by a knot one crossing change away from $K$. It is unknown whether there exists a complete Alexander neighbor with nontrivial Alexander polynomial. We eliminate infinite families of knots with nontrivial Alexander polynomial from having this property and discuss possible strategies for unresolved cases.

    Additionally, we use a condition on determinants of knots one crossing change away from unknotting number one knots to improve KnotInfo's \cite{knotinfo} unknotting number data on 11 and 12 crossing knots. Lickorish introduced an obstruction to unknotting number one in \cite{lickorish}, which proves the same result. However, we show that Lickorish's obstruction does not subsume the obstruction coming from the condition on determinants.
\end{abstract}

\maketitle

\section{Introduction}

Unknotting number and Alexander polynomials are classical knot invariants, so it is natural to consider the interaction between crossing changes of a knot $K$ and the Alexander polynomial $\triangle_K(t)$. Gordian distance, the minimal number of crossing changes needed to change one knot to another, is a generalization of unknotting number. In 1978, Kondo proved that there exists a knot with unknotting number one realizing any given Alexander polynomial \cite{kondo}. A natural next question to ask is whether there exists a nontrivial Alexander polynomial such that, given any second Alexander polynomial, there exist a pair of knots with Gordian distance one realizing the two polynomials. In 2012, Kawauchi proved that this is the case for Alexander polynomials of slice type (Corollary 5.2 in \cite{kawauchi}). Jong's problem asks whether there exists a pair of Alexander polynomials such that any two knots realizing the polynomials have Gordian distance at least 2. In 2012, Kawauchi found a family of pairs of polynomials for which this is the case \cite{kawauchi}. One example is the Alexander polynomials of the trefoil and figure eight knot. This brings us to a knot property that we will study in this paper:
\begin{definition}
    A knot $K$ is a \textbf{complete Alexander neighbor} if for any Alexander polynomial $p(t)$, there exists a knot $K'$ such that $K$ and $K'$ are one crossing change apart and $\triangle_{K'}(t)=p(t)$.
\end{definition}Since the Alexander polynomial is multiplicative under connected sum of knots, Kondo's result stating that there exists a knot with unknotting number one realizing any given Alexander polynomial implies that any knot with trivial Alexander polynomial is a complete Alexander neighbor \cite{kondo}. However, it is unknown whether any knot with nontrivial Alexander polynomial has this property.

\begin{question}\label{open} (Raised on pg. 1017 of \cite{nakanishiokada})
Does there exist a complete Alexander neighbor $K$ with nontrivial Alexander polynomial?
\end{question}

While Kawauchi found \emph{polynomials} which are realized by knots with particular Gordian distances, this question asks for a \emph{knot} whose Gordian neighbors realize all Alexander polynomials.

We can obstruct knots from this property in a variety of ways. One is by considering the algebraic unknotting number, introduced by Murakami in \cite{murakami}, which is the minimal number of crossing changes necessary to transform a knot into a knot with Alexander polynomial one. Of course, a complete Alexander neighbor must have algebraic unknotting number one. The database Knotorious eliminates 1,526 knots of the 2,977 prime knots with 12 crossings or fewer from being complete Alexander neighbors using algebraic unknotting number \cite{borodzik2}.

The following theorem gives an improvement on work by Nakanishi and Okada \cite{nakanishiokada}. See Section 2 for a definition of Nakanishi index.

\begin{theorem}
\label{nakanishiindexdet}
Let $K$ be a knot with Nakanishi index 1, where $\det(K)\geq 3$ and where $\det (K)$ is composite or $\det (K) \equiv 1 \mod 4$. Then $K$ is not a complete Alexander neighbor.
\end{theorem}

In addition, we can characterize the knots eliminated by Kawauchi in Proposition \ref{kawauchisets} \cite{kawauchi} (see Theorem \ref{breadth2thm} in Section 2). This result together with Theorem \ref{nakanishiindexdet} yields the following corollary.

\begin{corollary}
\label{breadthcor}
    Let $K$ be a knot whose Alexander polynomial $\triangle_{K}(t)$ has breadth 2. Then $K$ is not a complete Alexander neighbor.
\end{corollary}

Although we can eliminate infinitely many knots from Question \ref{open} using these methods, there are small knots which are not eliminated including $6_2$, $7_6$, $8_4$, $8_6$, $8_7$, and $8_{14}$. We do eliminate 2,528 of the 2,977 prime knots with crossing number 12 or less.

In their work on the relationship between crossing changes and Alexander polynomials, Nakanishi and Okada proved in \cite{nakanishiokada} a condition on the determinants $|\triangle_K(-1)|$ and $|\triangle_{K'}(-1)|$ of knots $K$ and $K'$ one crossing change apart where $K$ has unknotting number one. This gives a new obstruction to unknotting number one, which improves the KnotInfo \cite{knotinfo} data for five knots in the following theorem.

\begin{theorem}\label{unknottingnumberresult}
    The knots $11n_{162}, 12n_{805}, 12n_{814},12n_{844},$ and $12n_{856}$ have unknotting number greater than one.
\end{theorem}

We also give a second proof of this theorem using an obstruction by Lickorish in \cite{lickorish}. The two obstrutions are different; in particular, there are 17 examples of knots with 11 to 13 crossings where Lickorish's obstruction does not apply, but Nakanishi and Okada's condition on determinants obstructs unknotting number one. In addition, there is always the possibility expand the search to apply this obstruction from the condition on determinants, but we can determine from a single diagram of a knot whether or not Lickorish's obstruction applies.

\subsection{Acknowledgements}

We are extremely grateful to Mark Brittenham and Alex Zupan for their mentorship, advice, and support. We are grateful to Charles Livingston for his interest in this project and for helpful conversations. This work was completed while the author was a guest at the Max Planck Institute for Mathematics in Bonn and we are extremely grateful to MPIM for its support and hospitality. The author was partially supported by NSF grant DMS-2005518.

\section{The Relationship Between the Alexander Polynomial and Gordian Distance}

 Now we will find families of knots which are not complete Alexander neighbors. First we need to introduce some definitions.

 \begin{definition}
     The \textbf{algebraic unknotting number} $u_a(K)$ of a knot $K$ is the minimal number of crossing changes necessary to change $K$ to a knot with trivial Alexander polynomial.
 \end{definition}

Notice that any knot with algebraic unknotting number greater than one is not a complete Alexander neighbor since the trivial Alexander polynomial is not realized by any Gordian neighbor.
 
\begin{definition}
    The \textbf{Nakanishi index} $n(K)$ of a knot $K$ is the minimal $n$ such that the Alexander module of $K$ (the first homology of the infinite cyclic cover of the exterior of $K$, viewed as a $\mathbb{Z}[t,t^{-1}]$ module) is presented by an $n\times n$ matrix \cite{nakanishiindex}.
\end{definition}

Every crossing change can be described with a $\pm 1$ surgery along a loop around the crossing with linking number $0$ with the knot, so for any knot there exists a collection of these surgeries describing crossing changes resulting in the unknot. Levine \cite{levine} and Rolfsen \cite{rolfsen} introduced a surgery view of the Alexander matrix, which Nakanishi and Okada also describe in Section 2 of \cite{nakanishiokada}, so it is always possible to build an $n\times n$ Alexander matrix where $n$ is the unknotting number. However, in some cases there exists a smaller Alexander matrix. Also notice that algebraic unknotting number is a lower bound for unknotting number since the unknot has trivial Alexander polynomial. Furthermore, Nakanishi index is a lower bound for algebraic unknotting number (see Section 4.1 of \cite{borodzik}), so \[n(K)\leq u_a(K)\leq u(K)\] for any knot $K$ where $u(K)$ is the unknotting number of $K$. Therefore we can restrict our investigation of Question \ref{open} to knots with Nakanishi index one.
 
 Now we will eliminate families of knots with Nakanishi index one. First, we need to prove a lemma.

\begin{lemma}
\label{quadresidues} Let $n\geq 3$ be an odd integer. Then $n$ is composite or $n\equiv 1 \mod 4$ if and only if there exists some integer $d$ such that both $d$ and $-d$ are quadratic nonresidues mod $n$. 
\end{lemma}

\begin{proof}
Let $n\geq 3$ be an odd integer and let $f:\mathbb{Z}_n\rightarrow \mathbb{Z}_n$ such that $f(x)=x^2$ for all $x \in \mathbb{Z}_n$, so the image of $f$ is the set of quadratic residues mod $n$ together with $0$.

First notice that for every nonzero $y$ such that there exists $x \in \mathbb{Z}_n$ where $f(x)=y$ (equivalently, every quadratic residue mod $n$), we have
\[f(n-x)=(n-x)^2=n^2-2nx+x^2\equiv x^2=f(x)=y \mod n.\]
Since $n$ is odd, $n-x\not\equiv x \mod n$, so at least two distinct elements of $\mathbb{Z}_n$ map to each quadratic residue. Therefore, at most half the nonzero elements of $\mathbb{Z}_n$ are quadratic residues.

Consider the case where $n$ is prime. Then by the law of quadratic reciprocity, if $n\equiv 3 \mod 4$, then the negative of a residue modulo $n$ is a nonresidue and the negative of a nonresidue is a residue, as desired. Also by the law of quadratic reciprocity, if $n\equiv 1 \mod 4$, then the negative of a residue modulo $n$ is a residue and the negative of a nonresidue is a nonresidue. Since at most half the nonzero elements of $\mathbb{Z}_n$ are quadratic residues and $n\geq 3$, there exists a nonzero quadratic nonresidue $d$, so $d$ and $-d$ are quadratic nonresidues mod $n$ as desired.

Otherwise, $n$ is composite, so $n=ab$ for some positive odd integers $a$ and $b$ both greater than one. Assume without loss of generality that $a\leq b$. First we will show that one of the following must be true
\begin{enumerate}[(a)]
    \item $1\leq b-a<b+a \leq \frac{ab}{2}$

    \item $a=b$
    
    \item $a=3$ and $b=5$
\end{enumerate}
We will assume (a) is false and show that (b) or (c) must be true. Let $1>b-a$ or $b+a>\frac{ab}{2}$. In the case where $b-a<1$ we have $b-1<a\leq b$, so (b) holds. 

In the case where $b+a>\frac{ab}{2}$, we have that $a<\frac{b}{\frac{b}{2}-1}$. Then for $b\geq 6$ we have
\[1<a<\frac{b}{\frac{b}{2}-1}\leq 3\]
which is impossible since $a$ is an odd integer, so $b<6$. Since $b$ is odd and greater than one, we have that $b=3$ or $b=5$. Since $1<a\leq b$ and $a$ is odd, in the case that $b=3$, (b) holds and in the case that $b=5$ either $a=3$ so (c) holds or $a=5$ so (b) holds.

Consider the case (a) where $1\leq b-a \leq \frac{ab}{2}$ and $1\leq a+b \leq \frac{ab}{2}$. Notice that $f(b-a)$, $f(a+b)$, $f(n-(b-a))$, and $f(n-(a+b))$ are all congruent to $a^2+b^2$ mod $n=ab$. Since $1\leq b-a<b+a\leq \frac{ab}{2}$, we have that \[1\leq b-a<b+a\leq\frac{ab}{2}=n-\frac{ab}{2}\leq n-(a+b)<n-(b-a)\leq n-1,\] so at least three distinct elements of $\mathbb{Z}_n$ map to the same quadratic residue $a^2+2ab+b^2$ in $\mathbb{Z}_n$. Therefore, strictly less than half of the nonzero elements of $\mathbb{Z}_n$ are quadratic residues mod $n$. Thus, there exists some quadratic nonresidue $d$, so $d$ and $-d$ are quadratic nonresidues mod $n$ as desired.

Consider the case (b) where $a=b$. Then we have $a\not\equiv 0 \mod n$ and $f(a)=a^2=n\equiv 0\mod n$, so strictly less than half of the nonzero elements of $\mathbb{Z}_n$ are quadratic residues mod $n$. Therefore, there exists some quadratic nonresidue $d$, so $d$ and $-d$ are quadratic nonresidues mod $n$ as desired.

Consider the case (c) where $a=3$ and $b=5$. Then notice that $2$ and $-2$ are quadratic nonresidues mod $15$ as desired.
\end{proof}

We will use this lemma to improve the following result by Nakanishi and Okada.

\begin{lemma}
\label{nakanishialexandermatrix}
(Propositions 5 and 6 in \cite{nakanishiokada}) Let $K$ be a knot and let $A_K(t)=(a_{ij}(t))_{1\leq i,j\leq n}$ be an Alexander matrix of $K$ such that
\begin{enumerate}[(a)]
    \item $a_{ij}(t)=a_{ij}(t^{-1})$ for all $1\leq i,j\leq n$ and
    \item $a_{ij}(1)=\begin{cases}
    1 &\text{if }i=j\\
    0 &\text{if }i\neq j
    \end{cases}$
\end{enumerate}
Then a Laurent polynomial $p(t)$ is the Alexander polynomial of some knot $K'$ one crossing change away from $K$ if and only if there exist Laurent polynomials $r_1(t),...,r_n(t)$, and $m(t)$ such that
\begin{enumerate}[(1)]
    \item $m(t)=m(t^{-1})$, $m(1)=\pm 1$, and $r_i(1)=0$ for all $1\leq i\leq n$, and
    
    \item $p(t)=\pm\det\begin{pmatrix}&&&r_1(t^{-1})\\
    &A_K(t)&&\vdots\\
    &&&r_n(t^{-1})\\
    r_1(t)&\dots&r_n(t)&m(t)\end{pmatrix}$
\end{enumerate}
\end{lemma}

We now have all the tools to prove Theorem \ref{nakanishiindexdet}.

\begin{proof}[Proof of Theorem \ref{nakanishiindexdet}]
Let $K$ be a knot with Nakanishi index 1, where $\det(K)\geq 3$ and $\det (K)$ is composite or $\det (K) \equiv 1 \mod 4$. Since knot determinants are odd, by Lemma \ref{quadresidues}, there exists some quadratic nonresidue $d\mod \det(K)$ such that $-d$ is also a quadratic nonresidue$\mod \det(K)$.
    
    Also notice that since $K$ has Nakanishi index 1, the 1 by 1 matrix $[\triangle_K(t)]$ is an Alexander matrix of $K$, which satisfies conditions (a) and (b) in Lemma \ref{nakanishialexandermatrix}.
    
    Let $K'$ be a knot one crossing change away from $K$. Then, by Lemma \ref{nakanishialexandermatrix}, we have $\triangle_{K'}(t)=\triangle_K(t)\cdot m(t)-r(t)\cdot r(t^{-1})$ for some $m(t), r(t)\in \mathbb{Z}[t,t^{-1}]$ such that $m(t^{-1})=m(t)$, $|m(t)|=1$, and $r(1)=0$. Then
    \begin{align*}
        \det(K')= |\triangle_{K'}(-1)|&=|\triangle_{K}(-1)\cdot m(-1)-(r(-1))^2|\\
        \det(K')&=|\pm\det(K)\cdot m(-1)-(r(-1))^2|\\
        \text{so  }\det(K')&=|a\det(K)-b^2|
    \end{align*}
    for some $a,b\in \mathbb{Z}$.
    
    Consider the case where $a\det(K)\geq b^2$. Then
    \begin{align*}
        \det(K')&=a\det(K)-b^2\\
        b^2&=-\det(K')+a\det(K)
    \end{align*}
    so $-\det(K')$ is a quadratic residue$\mod \det(K)$. Therefore, $-\det(K')\not\equiv d\mod \det(K)$ and $-\det(K')\not\equiv -d\mod \det(K)$, so $\det(K')\not\equiv |d|\mod \det(K)$.
    
    Otherwise, $a\det(K)<b^2$. Then
    \begin{align*}
        \det(K')&=b^2-a\det(K)\\
        b^2&=\det(K')+a\det(K)
    \end{align*}
    so $\det(K')$ is a quadratic residue$\mod \det(K)$. Therefore, $\det(K')\not\equiv d\mod \det(K)$ and $\det(K')\not\equiv - d\mod \det(K)$, so $\det(K')\not\equiv |d|\mod \det(K)$.
    
    Since the knot determinants are exactly the odd natural numbers, there exists an Alexander polynomial $p(t)$ such that $|p(-1)|\equiv |d|\mod \det(K)$. As argued above, this Alexander polynomial is not realized by any knot one crossing change away from $K$.
\end{proof}

Kawauchi also eliminated families of knots from being complete Alexander neighbors in the following result.

\begin{proposition}\label{kawauchisets}
(Corollary 4.2 from \cite{kawauchi}) Let $p$ be any prime number, and $n$, $\ell$ integers coprime to $p$. If $p$ is an odd prime, then assume that $p$ is coprime to $1-4n$ and that $1-4n$ is a quadratic nonresidue mod $p$. Consider a set of Alexander polynomials \[S_{p,n,\ell}=\{n(t+t^{-1})+1-2n\}\cup\{(n+\ell p^{2s+1})(t+t^{-1})+1-2(n+\ell p^{2s+1})| s \in \mathbb{N}_0\}\] and let $a, b \in S_{p,n,\ell}$ such that $a\neq b$. Then for any knots $K_a, K_b$ such that $\triangle_{K_a}=a$ and $\triangle_{K_b}=b$, we have that $K_a$ and $K_b$ must have Gordian distance at least two.
\end{proposition}

We can characterize the knots Kawauchi has shown not to be complete Alexander neighbors here. First notice that any Alexander polynomial of breadth 2 of a knot $K$ can be written in the form $\triangle_{K}(t)=n(t+t^{-1})+1-2n$ for some nonzero integer $n$. 

\begin{theorem}
\label{breadth2thm}
An Alexander polynomial of breadth 2, $q(t)=n(t+t^{-1})+1-2n$ is contained in $S_{p, n, \ell}$ for some $p, n,$ and $\ell$ as defined in Proposition \ref{kawauchisets} if and only if $1-4n$ is not a square.
\end{theorem}

\begin{proof}
Let $q(t)=n(t+t^{-1})+1-2n$ be an Alexander polynomial of breadth 2 for some $n \in \mathbb{Z}$.

Assume $1-4n$ is not a square. First notice that for all non-square $x$, there exist infinitely many primes $p$ such that $x$ is a quadratic nonresidue mod $p$. Since $1-4n$ is not a square, there exist infinitely many primes $p_i$ such that $1-4n$ is a quadratic nonresidue mod $p_i$, so there exists such a prime $p_k$ such that $|1-4n|<p_k$ and $n<p_k$, so $1-4n$ and $n$ are coprime to $p_k$. Therefore, $q(t)\in S_{p_k,n,\ell}$ for some $\ell$ as defined in Proposition \ref{kawauchisets}.

Assume $q(t)\in S_{p,n,\ell}$ for some $p, n,$ and $\ell$ as defined in Proposition \ref{kawauchisets}. Then notice that either $1-4n$ is a quadratic nonresidue mod $p$ where $p$ is prime or $p=2$ and $n$ is coprime to $p$, meaning that $n$ is odd.

In the case where $1-4n$ is a quadratic nonresidue mod $p$, then $1-4n$ must not be a square.

In the case where $n$ is odd, $1-4n\equiv 5\mod 8$. Notice that odd squares are congruent to 1 mod $8$ since $(2x+1)^2=4x^2+4x+1=4x(x+1)+1$ and $x(x+1)$ must be even for any positive integer $x$. Therefore, $1-4n$ is not a square.
\end{proof}

Theorems \ref{nakanishiindexdet} and \ref{breadth2thm} yield Corollary \ref{breadthcor}.

\begin{proof}[Proof of Corollary \ref{breadthcor}]
Let $\triangle_K(t)=n(t+t^{-1})+1-2n$ be an Alexander polynomial of breadth 2 of a knot $K$ for some nonzero $n \in \mathbb{Z}$. In the case where $u_a(K)>1$, $K$ is not a complete Alexander neighbor, so we may assume $n(K)\leq u_a(K)=1$. Since $\triangle_K(t)$ has breadth 2 and so is nontrivial, we have $n(K)=1$. Notice that \[\det K=\begin{cases}
    1-4n &n<0\\
    -1+4n &n>0
\end{cases},\] so in the case where $n$ is negative, $\det K\equiv 1 \mod 4$ and $5\leq 1-4n$, meaning that $K$ is not a complete Alexander neighbor by Theorem \ref{nakanishiindexdet}. In the case where $n$ is positive, we have $1-4n<0$, so $1-4n$ is not a square. Therefore, $K$ is not a complete Alexander neighbor by Theorem \ref{breadth2thm} and Proposition \ref{kawauchisets}.
\end{proof}

Theorem 3 from \cite{kondo} by Kondo states that every Alexander polynomial is realized by a knot with unknotting number one and thus algebraic unknotting number one. Thus, Theorem \ref{nakanishiindexdet} proves that infinitely many knots are not complete Alexander neighbors. As an example, \[1+\sum_{i=1}^n\left((t^{2i}+t^{-2i})-(t^{2i-1}+t^{-2i+1})\right)\] for $n \in \mathbb{N}$ is an infinite class of Alexander polynomials with breadth $4n$ and determinant $1+4n$, so there exist infinitely many knots with unknotting number one, and thus algebraic unknotting number one and Nakanishi index one, realizing this class of Alexander polynomials which are eliminated from being a complete Alexander neighbor by Theorem \ref{nakanishiindexdet} and not by Corollary \ref{breadthcor} or their algebraic unknotting number.

Similarly, since every Alexander polynomial is realized by a knot with unknotting number one, Corollary \ref{breadthcor} proves that infinitely many knots are not complete Alexander neighbors. For example,
\[n(t+t^{-1})+1-2n\] for all nonzero integers $n$ is a collection of Alexander polynomials with breadth $2$ and determinant $|1-4n|$ including infinitely many prime determinants congruent to $3 \mod 4$, which are each realized by a knot with unknotting number one, and thus algebraic unknotting number one. Therefore, infinitely many knots are eliminated by Corollary \ref{breadthcor} and not by Theorem \ref{nakanishiindexdet} or their algebraic unknotting number.

Together, these three methods of proving that a knot is not a complete Alexander neighbor applies to all knots which meet at least one of the following criteria:
\begin{enumerate}[(a)]
    \item has algebraic unknotting number greater than one (which applies to 1,546 of the 2,977 prime knots with crossing number 12 or less),

    \item has determinant which is composite or congruent to $1$ mod $4$ (which applies to 2,392 of the 2,977 prime knots with crossing number 12 or less), or

    \item has Alexander polynomial of breadth 2 (which applies to 36 of the 2,977 prime knots with crossing number 12 or less).
\end{enumerate}

All together, this eliminates 2,528 of the 2,977 prime knots with 12 crossings or fewer. There are many very small candidates for a complete Alexander neighbor with nontrivial Alexander polynomial which are not yet eliminated. Through eight crossings these are $6_2$, $7_6$, $8_4$, $8_6$, $8_7$, and $8_{14}$.

All knots $K$ for which it is unresolved whether $K$ is a complete Alexander neighbor have algebraic unknotting number one and thus Nakanishi index one, meaning that $(\triangle_{K}(t))$ is an Alexander matrix for $K$, so we can restate the question of whether $K$ is a complete Alexander neighbor (see Propositions 5 and 6 in \cite{nakanishiokada}). $K$ is not a complete Alexander neighbor if and only if there exists an Alexander polynomial $p(t)$ such that for all Laurent polynomials $r(t)$ where $r(1)=0$, we have \begin{align*}
    p(t)&\not\equiv r(t) r(t^{-1})\mod \triangle_{K}(t)\text{ and}\\
    p(t)&\not\equiv - r(t) r(t^{-1})\mod \triangle_{K}(t).
\end{align*} This may give a new approach to Question \ref{open}.

For another possible method of investigating knots for which it is unresolved whether they are a complete Alexander neighbor, consider those with monic Alexander polynomial.  For example, the knots $6_2$, $7_6$, and $8_7$ have monic Alexander polynomial, which means we can use Nakanishi and Okada's algorithm used on the knots $3_1$ and $4_1$ in \cite{nakanishiokada} and the knots $5_1$ and $10_{132}$ in \cite{nakanishi} to characterize the set of Alexander polynomials realized by their Gordian neighbors. This characterization might be useful to determine whether this set includes all Alexander polynomials.

\section{Obstructions to Unknotting Number One}
We can leverage the effect of a single crossing change in a knot $K$ on the determinant or on the double cover $M_K$ of $S^3$ branched over $K$ to obtain two obstructions to unknotting number one. One was described by Lickorish in 1985 \cite{lickorish}. Another follows from work by Nakanishi and Okada in 2012 \cite{nakanishiokada}. Using these obstructions, we can show that many knots have unknotting number greater than one through simple computations, where the two obstructions are similar, though neither subsumes the other.

\begin{lemma}\label{detcondition} (Proposition 13 in \cite{nakanishiokada}) Let $K$ be a knot and $K'$ be a knot one crossing change away from $K$. If $K$ has unknotting number 1, then $\pm\det K' \equiv -n^2 \mod \det K$ for some integer $n$.
\end{lemma}

Therefore, by the contrapositive of Lemma \ref{detcondition}, given any knot $K$ where there exists a knot $K'$ one crossing change away such that $\mp\det K'$ is a quadratic nonresidue mod $\det K$, we have that $K$ has unknotting number greater than one.

Note that it is necessary for both $\det K'$ and $-\det K'$ to be a quadratic nonresidue mod $\det K$ to conclude that $K$ has unknotting number greater than one. For example, $3_1$ and $5_2$ are unknotting number one knots one crossing change apart with determinants $3$ and $7$ respectively. We have that $3$ is a quadratic nonresidue mod $7$, but $-3\equiv 4 \mod 7$ is a quadratic residue mod $7$. We also have that $-7\equiv 2 \mod 3$ is a quadratic nonresidue mod $3$ and $7\equiv 1\mod 3$ is a quadratic residue mod $3$.

In the KnotInfo database \cite{knotinfo}, we can use this observation to show that $11n_{162}, 12n_{805}$, $12n_{814},12n_{844},$ and $12n_{856}$ have unknotting number greater than one, where this was previously unknown in the database. This shows that $11n_{162}$ has unknotting number 2 and constrains the others to be $2$ or $3$.

\begin{figure}[h]
    \centering
    \includegraphics[height=10.9cm]{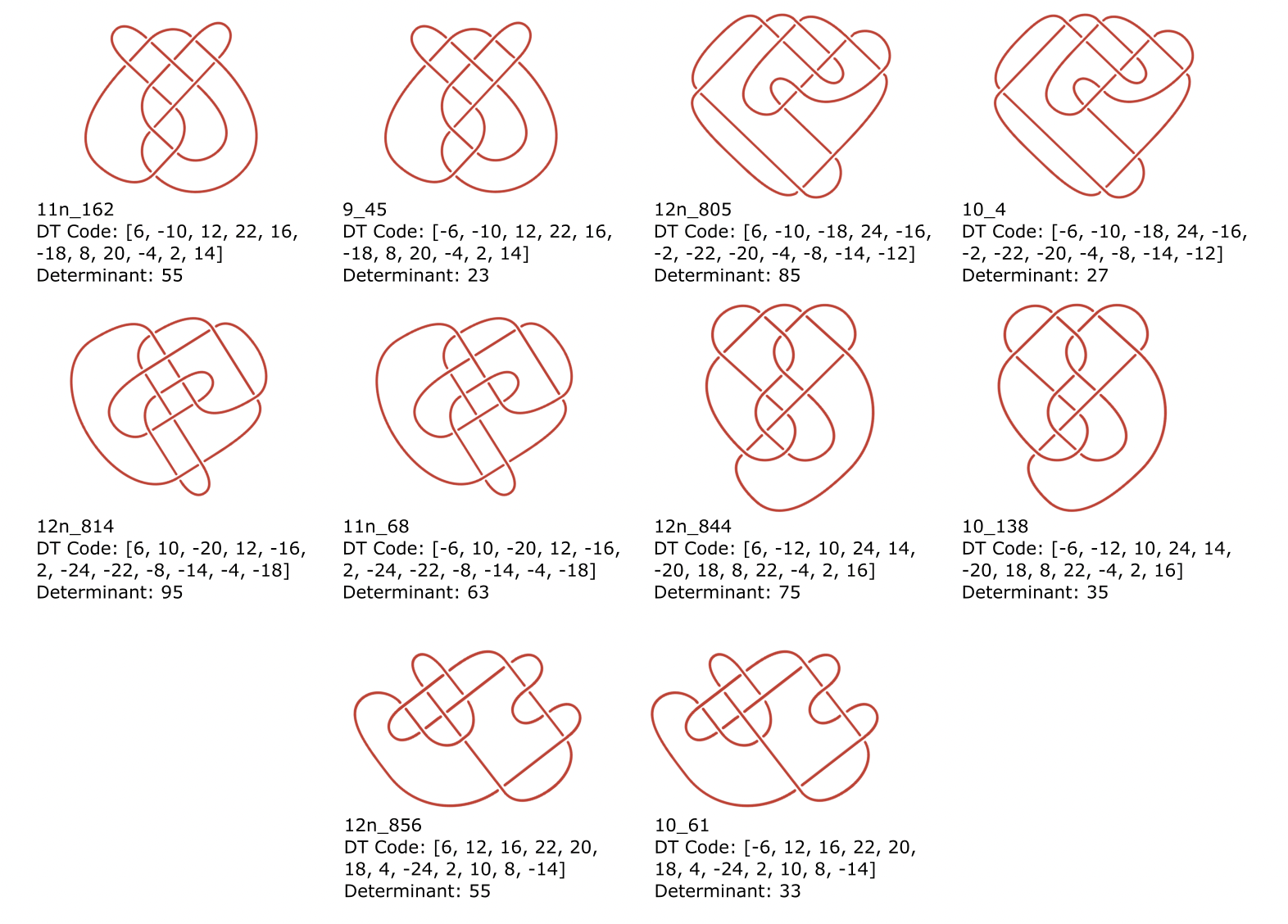}
    \caption{The knots $11n_{162}, 12n_{805}, 12n_{814},12n_{844},$ and $12n_{856}$ along with a knot one crossing change away from each of these. Under each knot is their name, a DT Code for the pictured diagram, and the knot's determinant.}
    \label{fig:knots}
\end{figure}

\begin{proof}[Proof of Theorem \ref{unknottingnumberresult}]
The knot $11n_{162}$ has determinant $55$ and DT code

\noindent$[6, -10, 12, 22, 16, -18, 8, 20, -4, 2, 14]$ in KnotInfo \cite{knotinfo}. We can change the sign of the first entry in the DT code to obtain $9_{45}$, a knot one crossing change away from $11n_{162}$. The determinant of $9_{45}$ is $23$. Since $23$ and $-23$ are both quadratic nonresidues mod $55$, by Lemma \ref{detcondition}, $11n_{162}$ has unknotting number greater than one.

In Figure \ref{fig:knots}, we see a knot one crossing change away from $11n_{162}, 12n_{805}, 12n_{814},12n_{844},$ and $12n_{856}$ whose determinant satisfies the contrapositive of Lemma \ref{detcondition}. Therefore, using a similar argument to the one above for $11n_{162}$, we conclude the proof.
\end{proof}

We identify these knots by performing a search with code using SnapPy \cite{snappy} in Sage to compute the determinant of each knot for which it is unknown in KnotInfo \cite{knotinfo} if the unknotting number is one, and compute the determinant of each knot obtained by changing the sign of one number in the DT code recorded in KnotInfo \cite{knotinfo}. Then we check whether the determinants satisfy the condition in Lemma \ref{detcondition}.

We can also use Lickorish's obstruction to show that these knots do not have unknotting number one \cite{lickorish}. To describe Lickorish's obstruction, we need to introduce some definitions.

\begin{definition}
Let $M$ be an oriented $3$-manifold where $H_1(M)$ is finite. Then the \textbf{linking form} of $M$ is $\lambda: H_1(M)\times H_1(M)\rightarrow \mathbb{Q}/\mathbb{Z}$ as defined below. Let $[\alpha], [\beta] \in H_1(M)$ represented by 1-cycles $\alpha$ and $\beta$ in $M$ respectively. Then $n\alpha$ bounds a disk $D$ for some integer $n$. Define $\lambda([\alpha],[\beta])=\frac{1}{n}i(D,\beta)$ where $i(D,\beta)$ is the intersection number of $D$ and $\beta$.
\end{definition}

\begin{definition}
Let $D$ be a connected, checkerboard colored diagram of a knot $K$. Let $R_0, R_1, ..., R_n$ be the white regions of $D$. Assign each crossing of $D$ a sign $\pm 1$ as in Figure \ref{fig:goeritz}. Let $g_{ij}$ be the sum of the signs of the crossings abutted by the white regions $R_i$ and $R_j$ for $0\leq i,j\leq n$ where $i\neq j$ and let $g_{ii}=-\displaystyle\sum_{i\neq j}g_{ij}$. A \textbf{Goeritz matrix} $G_K$ of $K$ is the $n\times n$ matrix $(g_{ij})_{1\leq i,j\leq n}$. Note that this eliminates all $g_{ij}$ where $i=0$ or $j=0$.
\end{definition}

\begin{figure}[h!]
    \centering
    \includegraphics[height=3cm]{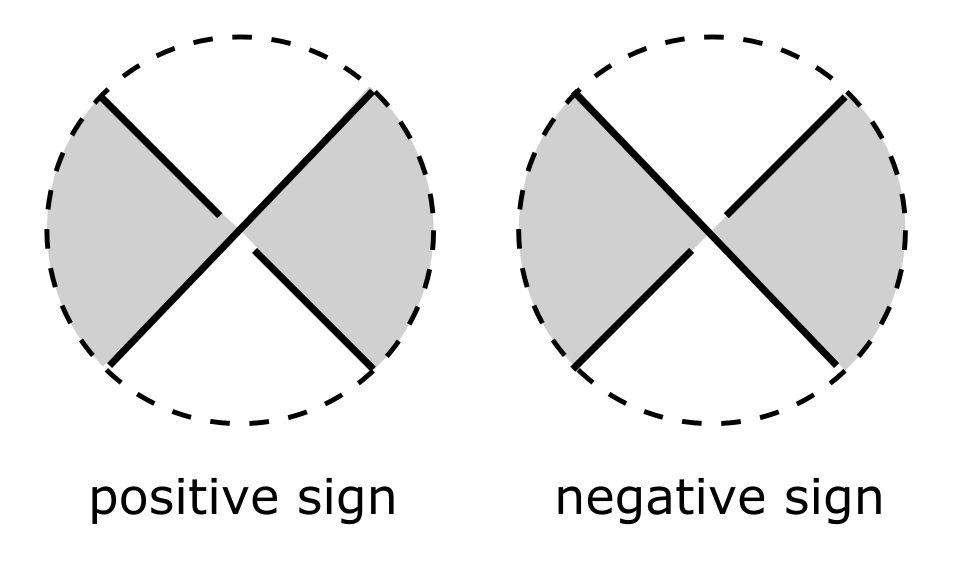}
    \caption{These are the sign conventions used in the definition of a Goeritz matrix.}
    \label{fig:goeritz}
\end{figure}

\begin{lemma} \label{lickorishlemmas} (Lemmas 1 and 2 in \cite{lickorish}) If $K$ is a knot with unknotting number one, then $M_K$ is obtained by $\pm\frac{\det K}{2}$-surgery on a knot in $S^3$ and $H_1(M_K)$ is cyclic with a generator $g$ such that $\lambda (g,g)=\frac{2}{\det K}\in \mathbb{Q}/\mathbb{Z}$.
\end{lemma}

\begin{lemma} \label{inversegoeritz} (page 253 in \cite{zhang}, page 761 of \cite{stoimenow}) Given a knot $K$, the linking form $\lambda$ of $M_K$ is given by $\pm (G_K)^{-1}$, meaning that $\lambda(g_i,g_j)=\pm(G_K^{-1})_{i,j}$ in $\mathbb{Q}/\mathbb{Z}$ where $\{g_1,g_2,...,g_n\}$ is a generating set of $H_1(M_K)$.
\end{lemma}

\begin{proof}[Alternate proof of Theorem \ref{unknottingnumberresult}]
First notice that in SnapPy, we can see
\begin{align*}
&G_{11n_{162}}^{-1}=\begin{pmatrix}\frac{16}{55}&\frac{8}{55}&\frac{1}{5}&\frac{3}{55}&\frac{6}{55}\\[6pt]\frac{8}{55}&\frac{4}{55}&\frac{3}{5}&\frac{29}{55}&\frac{3}{55}\\[6pt]\frac{1}{5}&\frac{3}{5}&\frac{1}{5}&\frac{3}{5}&\frac{1}{5}\\[6pt]\frac{3}{55}&\frac{29}{55}&\frac{3}{5}&\frac{4}{55}&\frac{8}{55}\\[6pt]\frac{6}{55}&\frac{3}{55}&\frac{1}{5}&\frac{8}{55}&\frac{16}{55}\end{pmatrix}, G_{12n_{805}}^{-1}=\begin{pmatrix}\frac{4}{17}&\frac{3}{17}&\frac{2}{17}&\frac{4}{17}\\[6pt]\frac{3}{17}&\frac{41}{85}&\frac{16}{85}&\frac{32}{85}\\[6pt]\frac{2}{17}&\frac{16}{85}&-\frac{29}{85}&\frac{27}{85}\\[6pt]\frac{4}{17}&\frac{32}{85}&\frac{27}{85}&\frac{54}{85}\end{pmatrix},\\ &G_{12n_{814}}^{-1}=\begin{pmatrix}\frac{36}{95}&\frac{3}{19}&-\frac{2}{95}&\frac{22}{95}&\frac{12}{95}\\[6pt]\frac{3}{19}&\frac{6}{19}&\frac{3}{19}&\frac{5}{19}&\frac{1}{19}\\[6pt]-\frac{2}{95}&\frac{3}{19}&-\frac{21}{95}&\frac{41}{95}&\frac{31}{95}\\[6pt]\frac{22}{95}&\frac{5}{19}&\frac{41}{95}&\frac{24}{95}&\frac{39}{95}\\[6pt]\frac{12}{95}&\frac{1}{19}&\frac{31}{95}&\frac{39}{95}&\frac{4}{95}\end{pmatrix}, G_{12n_{844}}^{-1}=\begin{pmatrix}\frac{4}{15}&\frac{1}{15}&\frac{1}{5}&\frac{2}{15}&\frac{2}{15}\\[6pt]\frac{1}{15}&\frac{1}{15}&\frac{3}{5}&\frac{8}{15}&\frac{2}{15}\\[6pt]\frac{1}{5}&\frac{3}{5}&\frac{1}{5}&\frac{3}{5}&\frac{1}{5}\\[6pt]\frac{2}{15}&\frac{8}{15}&\frac{3}{5}&\frac{1}{15}&\frac{1}{15}\\[6pt]\frac{2}{15}&\frac{2}{15}&\frac{1}{5}&\frac{1}{15}&\frac{4}{15}\end{pmatrix}, \text{ and}
\end{align*}
\begin{align*}
\hspace{-6.53cm}G_{12n_{856}}^{-1}=\begin{pmatrix}-\frac{14}{55}&\frac{27}{55}&\frac{2}{5}&\frac{19}{55}\\[6pt]\frac{27}{55}&\frac{54}{55}&\frac{4}{5}&\frac{38}{55}\\[6pt]\frac{2}{5}&\frac{4}{5}&\frac{4}{5}&\frac{3}{5}\\[6pt]\frac{19}{55}&\frac{38}{55}&\frac{3}{5}&\frac{41}{55}\end{pmatrix}.
\end{align*}

Assume for contradiction that $11n_{162}$ has unknotting number one. Then, by Lemma \ref{lickorishlemmas}, $H_1(M_{11n_{162}})$ is cyclic with a generator $g$ such that $\lambda (g,g)=\frac{2}{55}\in \mathbb{Q}/\mathbb{Z}$. Since $(G_{11n_{162}}^{-1})_{1,1}=\frac{16}{55}$, we have by Lemma \ref{inversegoeritz} that there exists $g_1\in H_1(M_{11n_{162}})$ such that $\lambda(g_1,g_1)=\pm \frac{16}{55}$. Since $H_1(M_{11n_{162}})$ is cyclic with a generator $g$, we have that $g_1=tg$ for some integer $t$, so \[\pm \frac{16}{55}=\lambda(g_1,g_1)=\lambda(tg,tg)=t^2\lambda(g,g)=t^2\frac{2}{55}\]
in $\mathbb{Q}/\mathbb{Z}$. Therefore $t^2\equiv \pm 8 \mod 55$, but $8$ and $-8$ are not a quadratic residues mod $55$, which is a contradiction.

Using a similar argument to the one above for $11n_{162}$, and the Goeritz matrices above, we conclude the proof.
\end{proof}

It is difficult to show that the first obstruction to unknotting number one does not apply to a particular knot since there are infinitely many crossing changes to check for the condition on determinants in Lemma \ref{detcondition}, however when we only check each crossing change done by a single sign change in the DT code for each knot recorded in KnotInfo \cite{knotinfo} up to 12 crossings, this obstruction shows that 1,273 knots have unknotting number greater than one out of 2,472 knots which are not known to have unknotting number one.

To show that Lickorish's obstruction does not apply to a particular knot $K$ we must check that $\lambda(g,g)/2$ or $-\lambda(g,g)/2$ is a quadratic residue for each $g \in H_1(M_K)$. In the case where $H_1(M_K)$ is not cyclic, we know by Lemma \ref{lickorishlemmas} that $K$ must have unknotting number greater than one and in the case where $H_1(M_K)$ is cyclic, checking the diagonal entries of $G_K$ is sufficient to determine whether Lickorish's obstruction is applicable to $K$. However, just checking that for each entry $(G_K^{-1})_{i,i}$ along the main diagonal of the inverse of the Goeritz matrix of $K$ for each prime $K$ with up to 12 crossings, shows that 1,269 knots have unknotting number greater than one out of 2,472 knots which are not known to have unknotting number one. We also have that 11 of the remaining knots which are not known to have unknotting number one have non-cyclic $H_1(M_K)$, so must have unknotting number greater than one.

In the prime knots up to 13 crossings, there are 17 examples ($11a_{47}$, $11n_{170}$, $12a_{166}$, $12a_{615}$, $12a_{886}$, $13a_{947}$, $13a_{1237}$, $13a_{1602}$, $13a_{1853}$, $13a_{1995}$, $13a_{2005}$, $13a_{2006}$, $13a_{2649}$, $13a_{4258}$, $13n_{1663}$, $13n_{2937}$, and $13n_{2955}$) where changing some crossing in the DT code recorded in KnotInfo \cite{knotinfo} yields a knot one crossing change away which satisfies the condition on determinants from Lemma \ref{detcondition} to show that the unknotting number must be greater than one, but Lickorish's obstruction does not apply using any of the diagonal entries of the inverse of the Goeritz matrix. However, all of these examples except $11n_{170}$ and $13a_{2649}$ have non-cyclic first homology of the double cover of $S^3$ branched over the knot, which also demonstrates that these knots have unknotting number greater than one. In the prime knots up to 13 crossings, there are 4 examples ($12n_{553}$, $13a_{1448}$, $13a_{2142}$, and $13n_{3264}$) where Lickorish's argument applies to one of the diagonal entries of the inverse of the Goeritz matrix, but no crossing change in KnotInfo's saved DT code \cite{knotinfo} gives a knot satisfying the condition on determinants from Lemma \ref{detcondition}.

Of course, we can use other methods to prove that many of these knots have unknotting number greater than one, but we see here that when we use the diagonal entries of the Goeritz matrix and the crossing changes from a sign change in the DT code, on small knots these obstructions are very similar, but not the same, and apply to many knots. Also, using the condition on determinants from Lemma \ref{detcondition} has the advantage that it is possible to expand the search to different crossing changes than those from sign changes in a particular DT code.

\iffalse

\fi

\bibliography{mybib}{}
\bibliographystyle{plain}

\end{document}